\documentclass[11pt,a4paper]{amsart}
\usepackage{amsfonts,amssymb,amscd,amsmath,enumerate,verbatim,calc}
\usepackage{mathrsfs}
\numberwithin{equation}{section}
\usepackage[pdftex]{hyperref}

\newtheorem{thm}{Theorem}[section]
\newtheorem{cor}[thm]{Corollary}
\newtheorem{lem}[thm]{Lemma}
\newtheorem{prop}[thm]{Proposition}
\newtheorem{defn}[thm]{Definition}

\newtheorem{rem}[thm]{Remark}

 \DeclareMathOperator{\Tor}{Tor}
\DeclareMathOperator{\Tot}{Tot}
\DeclareMathOperator{\Ext}{Ext} \DeclareMathOperator{\Supp}{Supp}
\DeclareMathOperator{\V}{V} \DeclareMathOperator{\Hom}{Hom}
\DeclareMathOperator{\Ker}{Ker} 
\DeclareMathOperator{\Image}{Im} 
\DeclareMathOperator{\cd}{cd} 

 \DeclareMathOperator{\Max}{Max}
\DeclareMathOperator{\lc}{H}

\DeclareMathOperator{\e}{E} \DeclareMathOperator{\T}{T}
\DeclareMathOperator{\h}{H}
\DeclareMathOperator{\pd}{pd}

\DeclareMathOperator{\Ass}{Ass}

\DeclareMathOperator{\dimSupp}{dimSupp}
\newcommand{\fa}{\mathfrak{a}}

\newcommand{\fm}{\mathfrak{m}}

\newcommand{\lo}{\longrightarrow}

\begin{document}
%%% ----------------------------------------------------------------------
%%% ----------------------------------------------------------------------
%%% ----------------------------------------------------------------------

\title[Cofiniteness  of generalized local
cohomology modules ]
{Cofiniteness  of generalized local
	cohomology modules with respect to a system of ideals}

%%% ----------------------------------------------------------------------
%%% ----------------------------------------------------------------------
%%% ----------------------------------------------------------------------
\bibliographystyle{amsplain}
%%% ----------------------------------------------------------------------
%%% ----------------------------------------------------------------------
%%% ----------------------------------------------------------------------

\author[M. Aghapournahr$^{1}$]{Moharram Aghapournahr$^{1}$}
\address{$^{1}$Department of Mathematics, Faculty of Science, Arak University,
Arak, 38156-8-8349, Iran.}
\curraddr{}
\email{m-aghapour@araku.ac.ir}

\author{TRAN TUAN NAM$^{2}$}
\address{$^{2}$Department of Mathematics-Informatics, Ho Chi Minh University of  Pedagogy, Ho Chi Minh city, Vietnam.}
\curraddr{}
\email{namtuantran@gmail.com, namtt@hcmue.edu.vn}

\author{NGUYEN THANH NAM$^{3}$}
\address{$^{3}$Faculty of Mathematics \& Computer Science, University of Science, VNU-HCM, Ho Chi Minh City, Vietnam}
\curraddr{} \email{hoanam8779@gmail.com}
%    author two information

\author{NGUYEN MINH TRI$^{*4,5}$}
\address{$^4$Vietnam National University, Ho Chi Minh City, Vietnam}

\address{$^5$Department of Mathematics and Physics, University of Information Technology, Ho Chi Minh City, Vietnam
}
\email{trinm@uit.edu.vn, triminhng@gmail.com}

\thanks{$^*$ Corresponding author}

\thanks{The last three authors are  funded by Vietnam National Foundation for Science and Technology Development (NAFOSTED) under Grant No.	101.04-2021.03}
\date{}
%\maketitle

% \author[Kamal Bahmanpour]{K. Bahmanpour}
%\address{ Department of Mathematics,  Faculty of Mathematical
%Sciences, University of Mohaghegh Ardabili, 56199-11367,  Ardabil, Iran.}
%\email{bahmanpour.k@gmail.com}

%\thanks{ This research of the author was in part supported by a grant from Arak university (No. 93-12224).}



%%% ----------------------------------------------------------------------
%%% ----------------------------------------------------------------------
%%% ----------------------------------------------------------------------

\begin{abstract}
Let $R$ be a commutative Noetherian ring, $\Phi$ a system of ideals of $R$ and $M,X$ two $R$-modules. In this paper, we study the  Artinianness and cofiniteness  of the module $H^i_{\Phi}(M,X)$ which is an extension of the generalized local cohomology modules of Herzog.
\end{abstract}

%and $I\in \Phi$. In this paper among other things we prove that if $M$ is
%finitely generated and $t\in \Bbb{N}$  such that the $R$-module
%$\lc^{i}_\Phi(M)$ is
%${\rm FD_{\leq 1}}$ (or weakly Laskerian) for all $i<t$, then $\lc^{i}_\Phi(M)$ is
%$\Phi$-cofinite for all $i<t$ and for any ${\rm FD_{\leq 0}}$ (or minimax)
%submodule $N$ of $\lc^t_\Phi(M)$, the $R$-modules
%$\Hom_R(R/I,\lc^t_\Phi(M)/N)$  and $\Ext^1_R(R/I,\lc^t_\Phi(M)/N)$ are finitely
%generated. Also it is shown that if $\cd I = 1$ or $\dim M/IM \leq 1$ {\rm(}e.g.,
%$\dim R/I \leq 1${\rm)} for all $I \in \Phi$,
%then the local cohomology module $\lc^{i}_\Phi(M)$ is $\Phi$-cofinite for all $i\geq 0$.
%These generalize  the main results of  Aghapournahr and Bahmanpour \cite{AB},
%Bahmanpour and Naghipour \cite{BN, BN2}. Also we study cominimaxness and
%weakly cofiniteness of local cohomology modules with respect to a system of
%ideals.

%%% ----------------------------------------------------------------------
%%% ----------------------------------------------------------------------
%%% ----------------------------------------------------------------------
\maketitle
%%% ----------------------------------------------------------------------
%%% ----------------------------------------------------------------------
%%% ----------------------------------------------------------------------

%%%%%%%%%%%%%%%%%%%%%%%%%%%%%%%%%%%%%%%%%%%%%%%%%%%%%%%%%%%%%%%%%%%%%%%%%%%%%%%%%%%%%%

\section{Introduction}

%%%%%%%%%%%%%%%%%%%%%%%%%%%%%%%%%%%%%%%%%%%%%%%%%%%%%%%%%%%%%%%%%%%%%%%%%%%%%%%%%%%%%%

Throughout this paper, $R$ is a commutative noetherian ring and $M$ is a finitely generated $R$-module.
Let $\Phi$ be a non-empty set of ideals of $R.$ We call $\Phi$ a system of ideals of $R$
if  $\mathfrak{a}, \mathfrak{b} \in \Phi,$ then there is an ideal $\mathfrak{c} \in\Phi$  such that $\mathfrak{c} \subseteq \mathfrak{a}\mathfrak{b}$ (see  \cite{BZ2}) For
every $R$-module $X,$ one can define
$$\Gamma_\Phi(X) = \{ x \in X \mid \mathfrak{a}x = 0 \text{ for some } \mathfrak{a} \in \Phi \}.$$
Then $\Gamma_\Phi(-)$ is an additive,
covariant, $R$-linear functor from the category
of $R$-modules to itself. Moreover, the functor $\Gamma_\Phi(-)$ is  left exact.  The functor $\Gamma_\Phi(-)$ is called the general local cohomology functor with respect to the system $\Phi$. For each $i \ge 0,$ the $i$th right derived functor of $\Gamma_\Phi(-)$ is denoted by $\lc^i_\Phi(-).$ For an ideal $\fa$ of $R,$ if $ \Phi = \{\fa^n\mid n \in\mathbb{N}_0\},$ then the functor $\lc^i_\Phi(-)$ coincides
with the ordinary local cohomology functor $\lc^i_\fa(-).$ Some basic properties of the module $\lc^i_\Phi(X)$ were shown in \cite{AB1}, \cite{AKS}, \cite{BZ2},  \cite{BZ3}  and \cite{hajcof}.

In \cite{BZ1}, the author introduced the bifunctor $\lc^i_\Phi(-,-)$ as follows: Let $M,X$ be two $R$-modules, the module $\lc_{\Phi}^i(M,X)$ is defined as
$$\lc_{\Phi}^i(M,X)\cong  \mathop {\lim }\limits_{\begin{subarray}{c}
	\longrightarrow  \\
	\fa\in \Phi
	\end{subarray}} \Ext_R^i(M/{\fa}M,X).$$
Then $\lc^i_\Phi( -,- )$ is an additive, $R$-linear functor which is contravariant in the first variable and covariant in the second variable. If $ \Phi = \{\fa^n\mid n \in\mathbb{N}_0\},$ then the bifunctor $\lc^i_\Phi(-,-)$ is naturally
equivalent to the bifunctor  $\lc^i_\fa(-,-)$ of Herzog in \cite{herkom}.

The cofiniteness and Artinianness of local cohomology modules are two of the most interesting properties which have been studied by many authors. In the case of local cohomology modules with respect to a system of ideals, the  $\fa$-cofinite modules of Hartshorne  was modified and called $\Phi$-cofinite modules in \cite{AKS}. We can see some conditions such that the module $H^i_\Phi(X)$ is $\Phi$-cofinite in  \cite{AB1} and \cite[4.8, 4.14, 4.15]{AKS}.  As an special case of  \cite[Definition 2.1]{Y} and generalization of FSF
modules (see \cite[Definition 2.1]{hung}), in \cite[Definition 2.1]{AB} the first author of present paper and Bahmanpour introduced the class of {\it{${\rm FD_{\leq n}}$}} modules. A module
$M$ is said to be ${\rm FD_{\leq n}}$ module, if there exists a finitely generated
submodule  $N$ of $M$ such that $\dim M/N\leq n$. Note that the class of ${\rm
	FD_{\leq -1}}$ is the same as finitely generated $R$-modules. Recall that an
$R$-module $M$ is called \emph{weakly Laskerian} if $\Ass_R(M/N)$ is a
finite set for each  submodule $N$ of $M$. The class of weakly Laskerian
modules is introduced in \cite{DiM}.
Recall also that a module $M$ is a \emph{minimax} module if there is a
finitely generated submodule $N$ of $M$ such that the quotient module $M/N$
is Artinian.  Minimax modules have been studied by
Z\"{o}schinger in
\cite{Zrmm}.

In this paper, we introduce the concept of $\Phi$-cofiniteness with respect to module $\lc^i_\Phi(M,X).$ In Section 2, we extend some results on the cofiniteness of the generalized local cohomology modules with respect to an ideal of Herzog. We prove in Theorem \ref{T:cdle1} that the module $\lc^i_\Phi(M,X)$ is $\Phi$-cofinite for all $i\ge 0$ if $\mathrm{cd}(\fa)\le 1$ for all $\fa\in \Phi$ (or $\cd \Phi\leq 1$). Next, we prove the following important and key proposition:

\begin{prop}{\rm(}See Proposition \ref{he}{\rm)}
	Let $R$ be a Noetherian ring, $\fa\in \Phi$ an ideal of $R$ and $M$ a finitely generated $R$-module.  Let $t\in\mathbb{N}_0$ be an integer and $X$ an $R$-module such that the $R$-modules $\Ext^i_R(R/\fa,X)$ are finite for all $i\leq t+1$ and the $R$-module $\lc^i_\Phi(M,X)$ is $\Phi$-cofinite for all
	$i<t$. Then the $R$-modules ${\rm
		Hom}_R(R/\fa,\lc^{t}_\Phi(M,X))$ and ${\rm Ext}^{1}_R(R/\fa,\lc^{t}_\Phi(M,X))$ are finitely generated.
\end{prop}

Using the above proposition we prove the following main result:

\begin{thm}{\rm(}See Lemma \ref{homext1}{\rm)}
	 Let $R$ be a Noetherian ring, $\fa\in \Phi$ an ideal of $R$ and $M$ a finitely generated $R$-module.  Let $t\in\mathbb{N}_0$ be an integer and $X$ an $R$-module such that the $R$-modules $\Ext^i_R(R/\fa,X)$ are finitely generated for all $i\leq t+1$ and the $R$-modules $\lc^i_\Phi(M,X)$ are ${\rm FD_{\leq1}}$ ~ $R$-modules for all $i<t$.
	Then, the following conditions hold:
	
	{\rm(i)} The $R$-modules $\lc^i_\Phi(M,X)$ are $\fa$-$ETH$-cofinite {\rm{(}}in particular $\Phi$-cofinite{\rm{)}} for all
	$i<t$.
	
	{\rm(ii)} For all ${\rm FD_{\leq 0}}$ {\rm{(}}or minimax{\rm{)}} submodule $N$ of
	$\lc^{t}_\Phi(M,X)$, the $R$-modules $${\Hom}_R(R/\fa,\lc^{t}_\Phi(M,X)/N)\,\,\, {\rm and}\,\,\,{\rm Ext}^1_R(R/\fa,\lc^{t}_\Phi(M,X)/N)$$ are finitely generated.
\end{thm}

As consequences of above main theorem we prove Corollaries \ref{T:minimaxcofinite}, \ref{nam}, \ref{cof}, \ref{wl1} and \ref{cof2}.

In Section 3, we study $\Phi$-cofinite Artinianness of $\lc^i_\Phi(M,X)$. More precisely we prove the following theorem:

\begin{thm}{\rm(}See Theorems \ref{T:pd+dim-IcofArtin} and \ref{moh}{\rm)}
	Let $\Phi$ be a system of ideals of a Noetherin ring $R$ and $M,X$ two finitely generated $R$-modules. Then the following statements hold:
	\begin{itemize}
	\item[(i)] If $(R,\fm)$ is a local ring and $\pd(M)<\infty$, then $\lc^{\mathrm{pd}(M)+\dim (R)}_{\Phi}(M,R)$ is $\Phi$-cofinite Artinian.
	\item[(ii)]	If $R$ is a semi-local and  $\dim R/{\fa}\leq0$ for all $\fa\in \Phi$, then $\lc^{i}_{\Phi}(M,X)$ is $\Phi$-cofinite Artinian for all $i\ge 0.$
	\end{itemize}	
\end{thm}

As a consequence of part (i) of above theorem, we prove the following result:

\begin{thm}{\rm(}See Theorem \ref{T:dimle2}{\rm)}
	Let $(R,\fm)$ be a local ring and $M,X$ two finitely generated $R$-modules with $\mathrm{pd}(M)<\infty.$ If $\dim(X)\le 2$ or $\dim (M)\le 2,$ then $\lc^{i}_{\Phi}(M,X)$ is $\Phi$-cofinite for  all $i\ge 0.$
\end{thm}

Throughout this paper, $R$ will always be a commutative Noetherian ring with
non-zero identity.  We denote $\{\frak p \in {\rm
	Spec}\,R:\, \frak p\supseteq \fa \}$ by $V(\fa)$. For any unexplained notation and terminology we refer the reader to
\cite{BH}, \cite{Mat} and \cite{BSh}.

%%%%%%%%%%%%%%%%%%%%%%%%%%%%%%%%%%%%%%%%%%%%%%%%%%%%%%%%%%%%%%%%

\section{Cofiniteness}

The following definition is a generalization of the concept $\fa$-cofinite modules which was introduced by Hartshorne.
\begin{defn}[\cite{AKS}]
	Let $\Phi$ be a system of ideals of $R$ and $X$ an $R$-module. The general local cohomology
	module $\lc^j_\Phi(X)$ is called to be $\Phi$-cofinite if there exists an ideal $\fa \in \Phi$ such that
	$\Ext^i_R(R/\fa, \lc^j_\Phi(X))$ is finitely generated for all $i.$
\end{defn}
In the case of module $\lc^i_\Phi(M,X),$ where $M$ is a finitely generated $R$-module, we have a similar definition.
\begin{defn}
	Let $\Phi$ be a system of ideals of $R$. The module $\lc^j_\Phi(M,X)$ is called to be $\Phi$-cofinite if there exists an ideal $\fa \in \Phi$ such that
	$\Ext^i_R(R/\fa, \lc^j_\Phi(M,X))$ is finitely generated for all $i.$
\end{defn}
The cohomological dimension with respect to ideal $\fa$ is denoted by $\mathrm{cd}(\fa)$ and defined as:
$$\mathrm{cd}(\fa)=\sup\{i\mid \lc^i_\fa(X)\ne 0 \text{ for any $R$-module } X\}.$$
\begin{defn}{\rm(}\cite{AB1}{\rm)}
	For any system of ideals $\Phi$ of $R$ and any $R$-module $X,$ we
	denote the cohomological dimension of $X$ with respect to $\Phi$ by $\mathrm{cd}(\Phi, X)$ and
	define as
	$$\mathrm{cd}(\Phi, X) = \sup\{i\mid \lc^i_\Phi(X) \ne  0\},$$
	if this supremum exists, otherwise, we define it as $-\infty.$
	
	We denote $\mathrm{cd}(\Phi, R)$ by $\mathrm{cd} (\Phi)$, therefore $\mathrm{cd}(\Phi) \le \sup\{\mathrm{cd}(\fa)\mid \fa\in \Phi\}.$
\end{defn}
The first result deals with the cohomological dimension with respect to an ideal. It is also an extension of \cite[2.2]{mafcf} and \cite[Theorem 1.1]{CGH}.
\begin{thm}\label{T:cdle1} Let $M,X$ be two finitely generated $R$-modules.
	Assume that $\mathrm{cd} (\fa)\le 1$ for all $\fa\in \Phi$   {\rm{(}}or $\cd(\Phi)\leq 1${\rm{)}}. Then $\lc^i_\Phi(M,X)$ is $\Phi$-cofinite for all $i\ge 0.$
\end{thm}
\begin{proof}
	Let $F=\Gamma_{\Phi}(-)$ and $G=\Hom_R(M,-)$ be functors from the category of $R$-modules to itself.
	For any $R$-module $K,$ by \cite[2.8]{hajcof} $$FG(K)=\Gamma_{\Phi}(\Hom_R(M,K))=\Hom_R(M,\Gamma_{\Phi}(K))=\Gamma_{\Phi}(M,K).$$ If $E$ is an injective $R$-module,  then so is
	$\Gamma_{\Phi}(E)$ by \cite[4.4]{AKS}. Let $\textbf{F}_\bullet$ be a free resolution of $M$
	$$\textbf{F}_\bullet: \cdots\to F_2\to F_1\to F_0\to M\to 0,$$
	where $F_i$ is finitely generated free for all $i\ge 0.$
	By applying the functor $\Hom_R(-,\Gamma_{\Phi}(E))$ to the free resolution of $M$, there is an exact sequence
	{\small $$0\to \Hom_R(M,\Gamma_{\Phi}(E))\to \Hom_R(F_0,\Gamma_{\Phi}(E))\to \Hom_R(F_1,\Gamma_{\Phi}(E))\to \cdots$$}and $\Hom_R(F_i,\Gamma_{\Phi}(E))$ is injective for all $i\ge 0.$
	Since $M$ is finitely generated, the exact sequence can be rewritten
	{\small $$0\to \Gamma_{\Phi}(\Hom_R(M,E))\to \Gamma_{\Phi}(\Hom_R(F_0,E))\to \Gamma_{\Phi}(\Hom_R(F_1,E))\to \cdots.$$}Note that $\Hom_R(\textbf{F}_\bullet,E)$ is an injective resolution of
	$\Hom_R(M,E).$ Then
	$R^iF(G(E))=\lc^i_{\Phi}(\Hom_R(M,E))=0$ for all $i>0.$  By \cite[10.47]{rotani}  there is a Grothendieck spectral sequence
	$$E_2^{p,q}=\lc^p_{\Phi}(\Ext^q_R(M,X))\underset{p}{\Rightarrow}\lc^{p+q}_{\Phi}(M,X).$$
Since $\mathrm{cd} (\fa)\le 1$ for all $\fa\in \Phi$   {\rm{(}}or $\cd(\Phi)\leq 1${\rm{)}}, we see that $E_2^{p,q}=0$ for all $q\ge 0,~ p>1.$ Let $n$ be a non-negative integer, there is  a filtration $\phi$ of submodules of $\lc^n=\lc^n_{\Phi}(M,X)$
	$$0=\phi^{n+1}\lc^n\subseteq \phi^n\lc^n\subseteq \ldots\subseteq \phi^1\lc^n\subseteq \phi^0\lc^n=\lc^n$$
	such that
	$$E_\infty^{i,n-i}\cong \phi^{i}\lc^n/\phi^{i+1}\lc^n$$
	for all $i\le n.$ Since $E_\infty^{p,q}$
	is a sub-quotient of $E_2^{p,q},$ it follows that $E_\infty^{p,q}=0$ for all $p>1.$ Consequently,
	$$\phi^2\lc^n=\phi^3\lc^n=\ldots=\phi^{n+1}\lc^n=0.$$
	Furthermore, $E_2^{1,n-1}=E_\infty^{1,n-1}\cong \phi^1\lc^n/\phi^2\lc^n=\phi^1\lc^n$ and $E_2^{0,n}=E_\infty^{0,n}\cong \phi^0\lc^n/\phi^1\lc^n.$ Hence, there is a short exact sequence
	$$0\to \lc^1_\Phi(\Ext_R^{n-1}(M,X))\to \lc^{n}_\Phi(M,X)\to \lc^0_\Phi(\Ext_R^n(M,X))\to 0.$$
	By \cite[2.15]{AB1}, the module $\lc^1_\Phi(\Ext_R^{n-1}(M,X))$ is $\Phi$-cofinite. Moreover, $\lc^0_\Phi(\Ext_R^n(M,X))$ is finitely generated, this implies that $\lc^{n}_\Phi(M,X)$ is $\Phi$-cofinite.
\end{proof}
\begin{cor}
	Let $M,X$ be two finitely generated $R$-modules. Assume that $\dim (R)\le 1.$ Then $\lc^i_\Phi(M,X)$ is $\Phi$-cofinite for all $i\ge 0.$
\end{cor}
\begin{proof}Under the conditions stated above, $\lc^i_\fa(X)=0$ for all $i>1$ and for all finitely generated $R$-module $X$  and $\fa\in \Phi.$ Hence $\mathrm{cd} (\fa)\le 1$ for all $\fa\in \Phi.$ The assertion follows from \ref{T:cdle1}.
\end{proof}

%%%%%%%%%%%%%%%%%%%%%%%%%%%%%%%%%%%%%%%%%%%%%%%%%%%%%%%%%%%%%%%%

The following two theorems are needed in the proof of the next  main results of this section.

\begin{thm}\label{v1}
Let $\Phi$ be a system of ideals of $R$ and $\fa\in \Phi$. Let $M$  be a finitely generated  $R$-module, $X$ an arbitrary $R$-module, and $t$ a non-negative integer such that
\begin{itemize}
\item[(i)] $\Ext_{R}^{t-i}(\Tor_{i}^{R}(R/\fa , M) , X)$ is finitely generated for all $i \leq t$, and
\item[(ii)] $\Ext_{R}^{t+1-i}(R/\fa , \lc^{i}_{\Phi}(M , X))$ is finitely generated for all $i \leq t$.
\end{itemize}
Then  $\Hom_{R}(R/\fa , \lc^{t}_{\Phi}(M , X))$  is finitely generated.
\end{thm}
\begin{proof}
Let
\begin{center}
$\e^{\bullet} = 0 \rightarrow X \rightarrow \e^{0} \rightarrow ... \rightarrow \e^{i} \rightarrow ...$
\end{center}
be an injective resolution of $X$ and apply $\Hom_{R}(M , \Gamma_{\Phi}(-))$  to its deletion $\e^{\bullet X}$ to get the complex
\begin{center}
$\Hom_{R}(M , \Gamma_{\Phi}(\e^{\bullet  X})) =  0  \rightarrow \Hom_{R}(M , \Gamma_{\Phi}(\e^{0})) \rightarrow ... \rightarrow
\Hom_{R}(M , \Gamma_{\Phi}(\e^{i})) \rightarrow ...$
\end{center}

Let
\begin{center}
$0 \rightarrow \Hom_{R}(M , \Gamma_{\Phi}(\e^{\bullet  X})) \rightarrow \T^{\bullet , 0} \rightarrow \T^{\bullet , 1}  \rightarrow ...
 \rightarrow \T ^{\bullet , i} \rightarrow ...$
\end{center}
be a Cartan-Eilenberg injective resolution of $\Hom_{R}(M , \Gamma_{\Phi}(\e^{\bullet  X}))$, which exists by \cite[Theorem 10.45]{rotani} and
consider the third quadrant bicomplex $\mathcal T = \lbrace \Hom_{R}(R/\fa , \T^{p,q}) \rbrace$. We denoted the total complex of $\mathcal T$ by
Tot($\mathcal T$).\\

The first filtration has $^{I} \e_{2}$ term the iterated homology $\h'^{p} \h''^{p,q}(\mathcal T)$. We have
\begin{align*}
\h''^{p,q}(\mathcal T) & \cong \h^{q}(\Hom_{R}(R/\fa , \T^{p,\bullet}))\\
                       & \cong \Ext^{q}_{R}(R/\fa , \Hom_{R}(M , \Gamma_{\Phi}(\e^{p}))\\
                       & \cong \Hom_{R}(\Tor_{q}^{R}(R/\fa , M) , \Gamma_{\Phi}(\e^{p}))
\end{align*}
from \cite[Lemma 4.4]{AKS} and \cite[Theorem 10.63]{rotani} therefore
\begin{align*}
^{I}\e_{2}^{p,q} & \cong \h'^{p}\h''^{p,q}(\mathcal T)\\
                           & \cong \h^{p}(\Hom_{R}(\Tor_{q}^{R}(R/\fa , M) ,\Gamma_{\Phi}(\e^{\bullet X}) )\\
                           & \cong \lc_{\Phi}^{p}(\Tor_{q}^{R}(R/\fa , M) , X)
\end{align*}
which yields, by \cite[Lemma 2.8 (ii)]{hajcof}, the third quadrant spectral sequence
\begin{center}
$^{I} \e _{2}^{p,q} := \Ext^{p}_{R}(\Tor_{q}^{R}(R/\fa , M) , X)  \underset{p}\Rightarrow \h^{p+q}(\Tot(\mathcal T)).$
\end{center}
For all $i \leq t$, we have $^{I}\e_{\infty}^{t-i,i} = ^{I}\e_{t+2}^{t-i,i}$ because, for all $j \geq t+2, ^{I}\e_{j}^{t-i-j,i+j-1} =0=~
^{I}\e_{j}^{t-i+j,i+1-j}$; so that $^{I}\e_{\infty}^{t-i,i}$ is finitely generated from the fact that $^{I}\e_{t+2}^{t-i,i}$ is a subquotient of
$^{I}\e_{2}^{t-i,i}$ which is finitely generated by assumption (i). There exists a finite filtration
\begin{center}
$0 = \phi^{t+1} \h^{t} \subseteq \phi^{t} \h^{t} \subseteq ... \subseteq \phi^{1} \h^{t} \subseteq \phi^{0} \h^{t} =\h^{t}(\Tot(\mathcal T))$
\end{center}
such that $^{I}\e_{\infty}^{t-i,i} \cong \phi^{t-i} \h^{t} / \phi^{t-i+1} \h^{t}$ for all $i\leq t$. Now the exact sequences
\begin{center}
$0 \rightarrow \phi^{t-i+1} \h^{t} \rightarrow \phi^{t} \h^{t-i} \rightarrow~ ^{I}\e_{\infty}^{t-i,i} \rightarrow 0,$
\end{center}
for all $i \leq t$, show that $\h^{t}(\Tot (\mathcal T))$ is finitely generated.

On the other hand, the second filtration has $^{II}\e_{2}$ term the iterated homology $\h''^{p}\h^{'q,p}(\mathcal T)$. Note that every short
exact sequence of injective modules splits and so it remains split after applying the functor $\Hom_{R}(R/\fa , -)$. By using this fact and
the fact that $\T^{\bullet , \bullet}$ is a Cartan-Eilenberg injective resolution of $\Hom_{R}(M , \Gamma_{\Phi}(\e^{\bullet X}))$, we get
\begin{align*}
\h'^{q,p}(\mathcal T) & \cong \h^{q}(\Hom_{R}(R/\fa , \T^{\bullet , p}))\\
                       & \cong \Hom_{R}(R/\fa , \h^{q}(\T^{\bullet ,p}))\\
                       & \cong \Hom_{R}(R/\fa , \h^{q,p}).
\end{align*}

Therefore
\begin{align*}
^{II}\e_{2}^{p,q} & \cong \h''^{p}\h'^{q,p}(\mathcal T)\\
                   & \cong \h^{p}(\Hom_{R}(R/\fa , \h^{q, \bullet}))\\
                   & \cong \Ext_{R}^{p}(R/\fa , \lc^{q}_{\Phi}(M , X))
\end{align*}
which gives the third quadrant spectral sequence
\begin{center}
$^{II}\e_{2}^{p,q} :=\Ext_{R}^{p}(R/ \fa , \lc^{q}_{\Phi}(M , X)) \underset{p}\Rightarrow  \h^{p+q}(\Tot(\mathcal T)). $
\end{center}

There exists a finite filtration
\begin{center}
$0 = \psi^{t+1} \h^{t} \subseteq \psi^{t} \h^{t} \subseteq ... \subseteq \psi^{1} \h^{t} \subseteq \psi^{0} \h^{t} =\h^{t}(\Tot(\mathcal T))$
\end{center}
such that $^{II}\e_{\infty}^{t-i,i} \cong \psi^{t-i} \h^{t} / \psi^{t-i+1} \h^{t}$ for all $i\leq t$. Since $\h^{t}(\Tot(\mathcal T))$ is finitely generated,
$^{II}\e_{\infty}^{0,t} \cong \psi^{0} \h^{t}/\psi^{1} \h^{t}$ is finitely generated. Therefore $^{II}\e_{t+2}^{0,t}$ is finitely generated because
$^{II}\e_{t+2}^{0,t} = ^{II}\e_{\infty}^{0,t}$ from the fact that $^{II}\e_{j}^{-j,t+j-1} = 0 = ^{II}\e_{j}^{j,t+1-j}$ for all
$j \geq t+2$. For all $r \geq 2$, let $^{II}Z^{0,t}_{r} = \Ker (^{II}\e_{r}^{0,t} \rightarrow ^{II}\e_{r}^{r,t+1-r})$ and
$^{II}B^{0,t}_{r} = \Image (^{II}\e_{r}^{-r,t+r-1} \rightarrow ^{II}\e_{r}^{0,t})$.
We have the exact sequences
\begin{center}
$0 \rightarrow ^{II}Z^{0,t}_{r} \rightarrow ^{II}\e^{0,t}_{r} \rightarrow ^{II}\e^{0,t}_{r} /  ^{II}Z^{0,t}_{r} \rightarrow 0$
\end{center}
and
\begin{center}
$0   \rightarrow ^{II}B^{0,t}_{r} \rightarrow  \rightarrow ^{II}Z^{0,t}_{r} \rightarrow ^{II}\e^{0,t}_{r+1} \rightarrow 0$.
\end{center}
Since $^{II}\e^{r,t+1-r}_{2}$ is finitely generated and $^{II}\e^{-r,t+r-1}_{2} = 0, ^{II}\e^{r,t+1-r}_{r}$ is finitely generated and $^{II}\e^{-r,t+r-1}_{r} = 0$ and so $^{II}\e^{0,t}_{r} /  ^{II}Z^{0,t}_{r}$ is finitely generated and $^{II}B^{0,t}_{r} = 0$. Hence $^{II}\e^{0,t}_{r}$ is finitely generated when
$^{II}\e^{0,t}_{r}$ is finitely generated. Therefore $^{II}\e^{0,t}_{2} = \Hom_{R}(R/\fa , \lc^{t}_{\Phi}(M , X))$ is finitely generated.
\end{proof}

\begin{thm} \label{v2}
Let $\Phi$ be a system of ideals of $R$ and $\fa\in \Phi$. Let $M$ be a finitely generated $R$-module, $X$ an arbitrary $R$-module, and $t$ a non-negative integer such that
\begin{itemize}
\item[(i)] $\Ext_{R}^{t+1-i}(\Tor_{i}^{R}(R/\fa , M), X)$ is finitely generated for all $i \leq t+1$, and
\item[(ii)] $\Ext_{R}^{t+2-i}(R/\fa , \lc_{\Phi}^{i}(M,X)) $is finitely generated for all $i < t$.
\end{itemize}
Then $\Ext_{R}^{1}(R/\fa , \lc_{\Phi}^{t}(M,X)) $is finitely generated.
\end{thm}
\begin{proof}
This is sufficiently similar to that of Theorem \ref{v1} to be omitted. We leave the proof to the
reader.
\end{proof}

%\begin{thm}
%Let $M$ be a finitely generated $R$–module, $X$ an arbitrary $R$–module, and $t$ a non-negative integer such that
%\begin{itemize}
%\item[(i)] $\Ext_{R}^{t+2-i}(\Tor_{i}^{R}(R/\fa , M), X)$ is finite for all $i \leq t+2$,
%\item[(ii)] $\Ext_{R}^{t+3-i}(R/\fa , \lc_{\Phi}^{i}(M,X)) $is finite for all $i < t$ and
%\item[(iii)] $\Hom_{R}(R/\fa , \lc_{\Phi}^{t+1}(M,X)) $is finite.
%\end{itemize}
%Then $\Ext_{R}^{2}(R/\fa , \lc_{\Phi}^{t}(M,X)) $.
%\end{thm}
%\begin{proof}
%The proof is similar to that of Theorem \ref{v1} and left to the reader.
%\end{proof}

\begin{lem}\label{equ}
	Let $R$ be a Noetherian ring, $\fa$ an ideal of $R$, $t$ a non-negative integer. Then for any $R$--modules
	$T$, the following conditions are equivalent:
	\begin{itemize}
		\item[(i)]  $\Ext^i_{R}(R/\fa, T)$ is finitely generated for all $0 \leq i\leq t$.
		\item[(ii)] for any finitely generated $R$--module $N$ with $\Supp_R(N)\subseteq \V(\fa)$,
		$\Ext^i_{R}(N, T)$ is finitely generated for all $0 \leq i\leq t$.
	\end{itemize}
\end{lem}
\begin{proof}
	See \cite[Lemma 1]{Ka}.
\end{proof}

%\begin{lem}
%	\label{FD0}
%	Let $\fa$ be an ideal of a Noetherian ring $R$ and $M$ be an  ${\rm FD_{\leq0}}$ {\rm{(}}or minimax{\rm{)}} $R$-module.
%	Then the
%	following statements are equivalent:
%	\begin{itemize}
%		\item[(i)] $M$ is $I$-$ETH$-cofinite.
%		\item[(ii)] The $R$-module $\Hom_R(R/I,M)$ is finitely generated.
%	\end{itemize}
%\end{lem}

%\begin{proof}
%	See  \cite[Theorem 5.3]{Mel3} and \cite[Theorem 2.1]{Mel}.
%\end{proof}

As a generalization of definition of cofinite modules with respect to an ideal (\cite[Definition
2.2]{Mel}), it is introduced in \cite[Definition 2.2]{A} the following definition.

\begin{defn}\label{def}
	An  $R$-module $M$ {\rm{(}}not necessary $\fa$-torsion{\rm{)}} is called
	$ETH$-cofinite with respect to ideal $\fa$ of $R$ or $\fa$-$ETH$-cofinite if
	$\Ext^{i}_{R}(R/\fa,M)$ is a finitely generated  $R$-module for all $i$.
\end{defn}

For more details about properties of this class see \cite[Example 2.3]{A}.

%%%%%%%%%%%%%%%%%%%%%%%%%%%%%%%%%%%%%%%%%

\begin{defn}
An $R$-module
$M$ is said to be ${\rm FD_{\leq n}}$ module, if there exists a finitely generated
submodule  $N$ of $M$ such that $\dim M/N\leq n$.	
\end{defn}

The following remark is some elementary properties of the class of ${\rm FD_{\leq n}}$ modules  which we shall use.

\begin{rem}\label{R:tinhchatminimax}
	There are some basic properties of the class of ${\rm FD_{\leq n}}$ modules:
	\begin{itemize}
		\item [(i)] The class of ${\rm FD_{\leq n}}$ modules contains all finitely generated, Artinian, minimax and weakly Laskerian modules.
		\item[(ii)] Let $0 \to L \to M \to N \to 0$ be a short exact sequence of $R$-modules. Then $M$ is a ${\rm FD_{\leq n}}$ module if and only
		if $L$ and $N$ are both ${\rm FD_{\leq n}}$ module. Thus, any submodule and quotient of a ${\rm FD_{\leq n}}$ module is ${\rm FD_{\leq n}}$.
		\item[(iii)] If $M$ is a finitely generated $R$-module and $L$ is a ${\rm FD_{\leq n}}$ $R$-module, then
		$\Ext^j_R(M, L)$ and $\Tor^R_j (M, L)$ are ${\rm FD_{\leq n}}$ for all $j \ge 0.$
	\end{itemize}
\end{rem}

%%%%%%%%%%%%%%%%%%%%%%%%%%%%%%%%%%%%%%%%%%%

In what follows, the next lemma plays key role.

%The following theorem is a generalization of \cite[Theorem 3.1]{AB} that in what
%follows the next theorem plays an important role.

\begin{lem} \label {wl}
	Let $R$ be a Noetherian ring and $\fa$ an ideal of $R$. Let $M$ be an ${\rm
		FD_{\leq 1}}$  $R$-module. Then $M$ is $\fa$-$ETH$-cofinite if and only if
	$\Hom_R(R/\fa,M)$ and  $\Ext^{1}_R(R/\fa,M)$ are finitely generated.
\end{lem}

\begin{proof}
	See \cite[Theorem 3.1]{AB1}.	
\end{proof}

%\begin{lem}
%	\label{FD1}
%	Let $I$ be an ideal of a Noetherian ring $R$ and $M$ be an  ${\rm FD_{\leq1}}$  $R$-module
%	such that $\Supp M\subseteq V(I)$. Then the following statements
%	are equivalent:
	
%	{\rm(i)} $M$ is  $I$-cofinite;
	
%	{\rm(ii)} The $R$-modules $\Hom_R(R/I,M)$ and $\Ext^1_R(R/I,M)$ are finitely generated.
%\end{lem}

The following proposition which is one of our main results of this section is a generalization of Theorems \cite[Theorem 2.1]{DY} and \cite[Theorem A]{DY1}.

\begin{prop}\label{he}
	Let $R$ be a Noetherian ring, $\fa\in \Phi$ an ideal of $R$ and $M$ a finitely generated $R$-module.  Let $t\in\mathbb{N}_0$ be an integer and $X$ an $R$-module such that the $R$-modules $\Ext^i_R(R/\fa,X)$ are finitely generated for all $i\leq t+1$ and the $R$-module $\lc^i_\Phi(M,X)$ is $\Phi$-cofinite for all
	$i<t$. Then the $R$-modules ${\rm
		Hom}_R(R/\fa,\lc^{t}_\Phi(M,X))$ and ${\rm Ext}^{1}_R(R/\fa,\lc^{t}_\Phi(M,X))$ are finitely generated.
\end{prop}
\proof
Using Lemma \ref{equ}, it is easy to see that the conditions (i) and (ii) of Theorems \ref{v1} and \ref{v2} are established. Thus the $R$-modules ${\rm
	Hom}_R(R/\fa,\lc^{t}_\Phi(M,X))$ and ${\rm Ext}^{1}_R(R/\fa,\lc^{t}_\Phi(M,X))$ are finitely generated by Theorems \ref{v1} and \ref{v2}.
\qed\\

We are now ready to state and prove the following main results of this section (Theorem
\ref{homext1} and   the Corollaries \ref{cof} and \ref{wl1}) which are extension of
Bahmanpour-Naghipour's results in \cite{BN}, \cite{BN2}, Quy's result in  \cite{hung}, \cite[Theorem 2.5]{C}, \cite[Theorem 3.6]{BSY1}  and \cite[Theorem 3.4 and Corollaries 3.5 and 3.6]{AB} and  \cite[Theorem 1.2]{CGH}.

\begin{thm}
	\label{homext1} Let $R$ be a Noetherian ring, $\fa\in \Phi$ an ideal of $R$ and $M$ a finitely generated $R$-module.  Let $t\in\mathbb{N}_0$ be an integer and $X$ an $R$-module such that the $R$-modules $\Ext^i_R(R/\fa,X)$ are finitely generated for all $i\leq t+1$ and the $R$-modules $\lc^i_\Phi(M,X)$ are ${\rm FD_{\leq1}}$ ~ $R$-modules for all $i<t$.
	Then, the following conditions hold:
	
	{\rm(i)} The $R$-modules $\lc^i_\Phi(M,X)$ are $\fa$-$ETH$-cofinite {\rm{(}}in particular $\Phi$-cofinite{\rm{)}} for all
	$i<t$.
	
	{\rm(ii)} For all ${\rm FD_{\leq 0}}$ {\rm{(}}or minimax{\rm{)}} submodule $N$ of
	$\lc^{t}_\Phi(M,X)$, the $R$-modules $${\Hom}_R(R/\fa,\lc^{t}_\Phi(M,X)/N)\,\,\, {\rm and}\,\,\,{\rm Ext}^1_R(R/\fa,\lc^{t}_\Phi(M,X)/N)$$ are finitely generated.
%	particular the sets $$\Ass_R({\rm %Hom}_R(R/I,\lc^{t}_\Phi(M,X)/N))\,\,\, {\rm
%		and}\,\,\,\Ass_R({\rm Ext}^1_R(R/I,\lc^{t}_\Phi(M,X)/N))$$ are  finite sets.
\end{thm}
\proof (i) We proceed by induction on $ t$. For the case $ t = 0$ there is nothing to prove. So, let $ t >0$ and the result
has been proved for smaller values of $t$. By the inductive assumption,
$\lc^{i}_\Phi(M,X)$ is $\fa$-$ETH$-cofinite for $ i = 0, 1, ..., t-2$. Hence by  Lemma \ref{he}  and
assumption, ${\rm Hom}_R(R/\fa,\lc^{t-1}_\Phi(M,X))$ and
${\rm Ext}^1_R(R/\fa,\lc^{t-1}_\Phi(M,X))$ are finitely generated. Therefore by Lemma
\ref{wl}, $\lc^{i}_\Phi(M,X)$ is $\fa$-$ETH$-cofinite for all $ i < t$. This completes the inductive step.\\
%Use induction on $i$ and corollary 2.3
%and Lemma 2.4.
(ii) In view of   (i) and Lemma \ref{he},  ${\rm Hom}_R(R/\fa,\lc^{t}_\Phi(M,X))\,\,\, {\rm
	and}\,\,\,{\rm Ext}^1_R(R/\fa,\lc^{t}_\Phi(M,X))$ are finitely generated. On the other hand,
according to \cite[Theorem 2.7]{AB1}, $N$ is $\fa$-$ETH$-cofinite. Now, the exact sequence

\begin{center}
	$0\longrightarrow N\longrightarrow  H^{t}_\Phi(M,X) \longrightarrow H^{t}_\Phi(M,X)/N
	\longrightarrow 0$
\end{center}

induces the following exact sequence,
$${\rm
	Hom}_R(R/\fa,\lc^{t}_\Phi(M,X))\longrightarrow {\rm
	Hom}_R(R/\fa,\lc^{t}_\Phi(M,X)/N)\longrightarrow {\rm Ext}^1_R(R/\fa,N)\longrightarrow$$ $$
{\rm Ext}^1_R(R/\fa,\lc^{t}_\Phi(M,X))\longrightarrow {\rm
	Ext}^1_R(R/\fa,\lc^{t}_\Phi(M,X)/N)\longrightarrow {\rm Ext}^2_R(R/\fa,N).$$

Consequently  $${\rm Hom}_R(R/\fa,\lc^{t}_\Phi(M,X)/N)\,\,\, {\rm and}\,\,\,{\rm
	Ext}^1_R(R/\fa,\lc^{t}_\Phi(M,X)/N)$$ are finitely generated, as required.\qed\\

%%%%%%%%%%%%%%%%%%%%%%%%%%%%%%%%%%%%%%%%%%%%%%%%%%%%%%%%%%%%%%%%

The following theorem illustrates a relationship between the minimaxness and the $\Phi$-cofiniteness. In \cite[2.3]{CGH}, if $\lc^i_\fa(M,X)$ is minimax for all $i<t,$ then $\lc^i_\fa(M,X)$ is $\fa$-cofinite for all $i<t.$  Now we extend this property in the case $\lc^i_\Phi(M,X).$
\begin{cor}\label{T:minimaxcofinite}
	Let $R$ be a Noetherian ring, $\fa\in \Phi$ an ideal of $R$ and $M$ a finitely generated $R$-module.  Let $t\in\mathbb{N}_0$ be an integer and $X$ an $R$-module such that the $R$-modules $\Ext^i_R(R/\fa,X)$ are finitely generated for all $i\leq t+1$ and the $R$-modules $\lc^i_\Phi(M,X)$ are  minimax for all $i<t.$  Then, the following conditions hold:
	
	{\rm(i)} The $R$-modules $\lc^i_\Phi(M,X)$ are $\fa$-$ETH$-cofinite {\rm{(}}in particular $\Phi$-cofinite{\rm{)}} for all
	$i<t$.
	
	{\rm(ii)} For all ${\rm FD_{\leq 0}}$ {\rm{(}}or minimax{\rm{)}} submodule $N$ of
	$\lc^{t}_\Phi(M,X)$, the $R$-modules $${\Hom}_R(R/\fa,\lc^{t}_\Phi(M,X)/N)\,\,\, {\rm and}\,\,\,{\rm Ext}^1_R(R/\fa,\lc^{t}_\Phi(M,X)/N)$$ are finitely generated.
\end{cor}

\begin{proof}
Use Theorem \ref{homext1} and note that the category of minimax
modules is contained in the category of ${\rm FD_{\leq1}}$ modules.
\end{proof}

\begin{cor}\label{nam}
Let $R$ be a Noetherian ring, $\fa\in \Phi$ an ideal of $R$ and $M$ a finitely generated $R$-module.  Let $t\in\mathbb{N}_0$ be an integer and $X$ an $R$-module such that the $R$-modules $\Ext^i_R(R/\fa,X)$ are finitely generated for all $i\leq t+1$. Assume that $\lc^i_{\Phi}(X)$ is ${\rm FD_{\leq 1}}$ module for all $i<t.$  Then $\lc^i_{\Phi}(M,X)$ is $\Phi$-cofinite  for all $i<t$ and $\Hom_R(R/\fa,\lc^t_{\Phi}(M,X))~ \text{and}~ {\rm Ext}^1_R(R/\fa, \lc^{t}_\Phi(M,X))$ are finitely generated.
\end{cor}
\begin{proof}
	Let $F=\Hom_R(M,-), G=\Gamma_\Phi(-)$ be functors from the category of $R$-modules to itself and $\fa\in \Phi.$ For any $R$-module $K,$ by \cite[Lemma 2.8]{hajcof}
	$$FG(K)=\Hom_R(M,\Gamma_\Phi(K))\cong \lc^0_\Phi(M,K).$$ Let $\e$ be an injective $R$-module, then $\Gamma_\Phi(\e)$ is injective by  \cite[Lemma 4.4]{AKS}. Hence $R^iF(G(\e))=0$ for all $i>0.$ By \cite[10.47]{rotani}, there is a Grothendieck spectral sequence
	$$E_2^{p,q}=\Ext_R^p(M,\lc^q_\Phi(X))\underset{p}{\Rightarrow}\lc^{p+q}_\Phi(M,X).$$
	Combining the hypothesis with \ref{R:tinhchatminimax}(iii),  $E_2^{p,q}$ is ${\rm FD_{\leq 1}}$ module for all $p\ge0,~q<t.$ Let $k<t,$ there is a filtration $\phi$ of submodules of $\lc^k=\lc^k_\Phi(M,X)$ such that
	$$E_\infty^{i,k-i}\cong \phi^i\lc^k/\phi^{i+1}\lc^k$$
	for all $i\le k.$ The module $E_\infty^{i,k-i}$ is ${\rm FD_{\leq 1}}$ module for all $0\le i\le k$ by \ref{R:tinhchatminimax}(ii) as $E_\infty^{i,k-i}$ is a sub-quotient of $E_2^{i,k-i}.$  Therefore $\lc^k_\Phi(M,X)$ is ${\rm FD_{\leq 1}}$ module for all $k<t$ and  the assertion follows from Theorem \ref{homext1}.
\end{proof}

%%%%%%%%%%%%%%%%%%%%%%%%%%%%%%%%%%%%%%%%%%%%%%%%%%%%%%%%%%%%%%%%

\begin{cor}
	\label{cof} Let $R$ be a Noetherian ring, $\fa\in \Phi$ an ideal of $R$ and $M$ a finitely generated $R$-module.  Let $X$ be an $R$-module such that the $R$-modules $\lc^i_\Phi(M,X)$ are ${\rm
		FD_{\leq1}}$ {\rm{(}}or weakly Laskerian{\rm{)}}~ $R$-modules for all $i$. Then,
	
	{\rm(i)} The $R$-modules $\lc^i_\Phi(M,X)$ are $\fa$-$ETH$-cofinite {\rm{(}}in particular $\Phi$-cofinite{\rm{)}} for all $i$.
	
	{\rm(ii)} For any $i\geq 0$ and for any ${\rm FD_{\leq 0}}$ {\rm{(}}or minimax{\rm{)}} submodule $N$ of
	$\lc^{i}_\Phi(M,X)$, the $R$-module $\lc^{i}_\Phi(M,X)/N$ is $\fa$-$ETH$-cofinite {\rm{(}}in particular $\Phi$-cofinite{\rm{)}}.
\end{cor}
\proof (i)  Follows by Theorem \ref{homext1} (i).\\ (ii)  In view of  (i) the $R$-module $\lc^i_\Phi(M,X)$ is $\fa$-$ETH$-cofinite for
all $i$. Hence the $R$-module ${\rm Hom}_R(R/\fa,N)$ is finitely generated, and so
it follows from \cite[Theorem 2.7]{AB1}  that $N$ is
$\fa$-$ETH$-cofinite. Now, the exact sequence

\begin{center}
	$0\longrightarrow N\longrightarrow  \lc^{i}_\Phi(M,X) \longrightarrow \lc^{i}_\Phi(M,X)/N
	\longrightarrow 0$
\end{center}

and \cite[Lemma
2.4]{AB1} implies that the $R$-module $\lc^{i}_\Phi(M,X)/N$ is  $\fa$-$ETH$-cofinite.\qed\\

%-------------------------------------------------------------

\begin{cor}
	\label{wl1}
	Let $\fa$ be an ideal of a Noetherian ring $R$, $M$ a finitely generated $R$-module and $X$ an $\fa$-$ETH$-cofinite $R$-module. Let
	$t\in\Bbb{N}_0$ such that the
	$R$-modules $\lc^i_\Phi(M,X)$ are weakly Laskerian for all $i<t$. Then, the following
	conditions hold:
	
	{\rm(i)} The $R$-modules $\lc^i_\Phi(M,X)$ are $\fa$-$ETH$-cofinite {\rm{(}}in particular $\Phi$-cofinite{\rm{)}}  for all $i<t$.
	
	{\rm(ii)} For all ${\rm FD_{\leq 0}}$ {\rm{(}}or minimax{\rm{)}} submodule $N$ of
	$\lc^{t}_\Phi(M,X)$, the $R$-modules $${\rm Hom}_R(R/\fa,\lc^{t}_\Phi(M,X)/N)\,\,\, {\rm
		and}\,\,\,{\rm Ext}^1_R(R/\fa,\lc^{t}_\Phi(M,X)/N)$$ are finitely generated. % In particular the
%	set $\Ass_R(\lc^{t}_I(M,X)/N)$ is  finite.
\end{cor}
\proof Use Theorem \ref{homext1} and note that the category of weakly Laskerian
modules is contained in the category of ${\rm FD_{\leq1}}$ modules.\qed\\

%---------------------------------------------------------------

The following corollary is a generalization of \cite[Corollary 2.7]{BN}, \cite[Theorem 2.5]{DH} and \cite[Theorem 1.2]{CGH}.

\begin{cor}
	\label{cof2} Let $R$ be a Noetherian ring, $\fa\in \Phi$ an ideal of $R$ and $M$ a finitely generated $R$-module.  Let $X$ be an  $\fa$-$ETH$-cofinite $R$-module  such that  $\dim X/{\fa}X \leq 1$ {\rm(}e.g.,
	$\dim R/\fa \leq 1${\rm)}. Then,
	\begin{itemize}
		\item[(i)] the $R$-modules $\lc^{i}_\Phi(M,X)$ are $\fa$-$ETH$-cofinite {\rm{(}}in particular $\Phi$-cofinite{\rm{)}}  for all $i$.
		\item[(ii)]  for any $i\geq 0$ and for any ${\rm FD_{\leq 0}}$ {\rm{(}}or minimax{\rm{)}} submodule $N$ of
		$\lc^{i}_\Phi(M,X)$, the $R$-module $\lc^{i}_\Phi(M,X)/N$ is   $\fa$-$ETH$-cofinite {\rm{(}}in particular $\Phi$-cofinite{\rm{)}}.
	\end{itemize}
\end{cor}
\begin{proof}
	(i) Since  by \cite[Lemma 2.1]{BZ1}, $$\lc^{i}_\Phi(M,X)\cong\underset { \fa \in \Phi}
		\varinjlim \lc^{i}_\fa(M,X),$$ it is easy to see that $\Supp_R(\lc^{i}_\Phi(M,X))
		\subseteq\underset{\fa \in \Phi}\bigcup \Supp_R(\lc^{i}_\fa(M,X))$ and therefore
		$$\dimSupp \lc^{i}_\Phi(M,X) \leq \sup \{ \dimSupp \lc^{i}_{\fa}(M,X)| \fa \in
		\Phi \} \leq 1,$$ thus $\lc^{i}_\Phi(M,X)$ is ${\rm FD_{\leq 1}}$ $R$-module
		and the assertion follows by  Corollary \ref{cof}
		(i).\\
		(ii) Proof is the same as \ref{cof} (ii).
\end{proof}

\begin{cor}
	\label{torextlc} Let $\Phi$ be a system of ideals of a Noetherian ring $R$, $M$ a finitely generated $R$-module and $X$ an $\fa$-$ETH$-cofinite $R$-module such that  $ \lc^{i}_\Phi(M,X)$ are ${\rm FD_{\leq 1}}${\rm{(}}or
	weakly Laskerian{\rm{)}}~ $R$-modules for all $ i \geq 0$. Then for
	each finitely generated $R$-module $N$, the $R$-modules ${\rm Ext}^{j}_R(N,\lc^{i}_\Phi(M,X))$ and
	${\rm Tor}^R_{j}(N,\lc^{i}_\Phi(M,X))$ are $\fa$-$ETH$-cofinite {\rm{(}}in particular $\Phi$-cofinite{\rm{)}} and ${\rm FD_{\leq1}}$ modules for
	all $ i \geq 0$ and $ j \geq 0$.
\end{cor}

\proof Follows by \cite[Corollary 2.11]{A}.
\qed\\

%%%%%%%%%%%%%%%%%%%%%%%%%%%%%%%%%%%%%%%%%%%%%%%%%%%%%%%%%%%%%%%

\section{Artinianess and cofiniteness}

It is well-known that $\lc^{\dim(X)}_\fa(X)$ is an $\fa$-cofinite Artinian (see \cite[Proposition 5.1]{Mel}) $R$-module and $\lc^{\mathrm{pd}(M)+\dim(X)}_\fa(M,N)$ is an $\fa$-cofinite Artinian $R$-module (see \cite[Proposition 2.7 and Theorem 2.8]{C}). Now, we consider the top generalized local cohomology module with respect to a system of ideals.
\begin{thm}\label{T:pd+dim-IcofArtin}
	Let $(R,\fm)$ be a local ring and $M$ a finitely generated $R$-module with $\mathrm{pd}(M)<\infty.$ Then $\lc^{\mathrm{pd}(M)+\dim (R)}_{\Phi}(M,R)$ is $\Phi$-cofinite Artinian.
\end{thm}

\begin{proof}Let $p=\mathrm{pd}(M)$ and $d=\dim (R).$ We consider the spectral sequence $$E_2^{p,q}=\Ext_R^p(M,\lc^q_\Phi(R))\underset{p}{\Rightarrow}\lc^{p+q}_\Phi(M,R).$$
	There is a filtration $\phi$ of submodules of $H^{p+d}=H^{p+d}_\Phi(M,R)$  such that
	$$E_\infty^{i,p+d-i}\cong \phi^i\lc^{p+d}/\phi^{i+1}\lc^{p+d}$$
	for all $i\leq p+d.$  By \cite[2.7]{BZ2},	$E_2^{i,j}=0$ for all $i>p$ or $j>d.$ Since $E_\infty^{i,p+d-i}$ is a sub-quotient of $E_2^{i,p+d-i}$, we can assert that $E_\infty^{i,p+d-i}=0$ for all $i\ne p.$ Consequently, $$\phi^0\lc^{p+d}=\ldots=\phi^{p}\lc^{p+d} \text{ and }\phi^{p+1}\lc^{p+d} =\ldots=\phi^{p+d+1}\lc^{p+d}=0.$$
	The homomorphisms of the spectral sequence $$0= E_2^{p-2,d+1}\to E_2^{p,d}\to E_2^{p+2,d-1}=0$$
	induces  that $E_2^{p,d}= E_\infty^{p,d}.$ Hence, there is an isomorphism $$\Ext_R^p(M,\lc^d_\Phi(R))\cong \phi^p\lc^{p+d}/\phi^{p+1}\lc^{p+d}=\lc^{p+d}_\Phi(M,R).$$
	The Artianness of  $\lc^{p+d}_\Phi(M,R)$ is  followed from \cite[Theorem 3.1]{BZ3}.
	
	Our next goal is to prove the cofiniteness by induction on $d.$ When $d=0,$ $R$ is an Artinian $R$-module and $\Supp(R)\subseteq\{\fm\}\subseteq \bigcup_{\fa\in \Phi}V(\fa).$ Hence $R$ is $\Phi$-torsion by \cite[Lemma 2.4]{khagen}. It follows from \cite[Lemma 2.8]{hajcof} that $$H^{p}_\Phi(M,R)\cong \Ext_R^{p}(M,R).$$
	Therefore, $\lc^p_\Phi(M,R)$ is finitely generated and then it is $\Phi$-cofinite.
	
	Assume that $d>0.$ Let $\overline{R}=R/\Gamma_{\Phi}(R)$ which is $\Phi$-torsion-free by \cite[Lemma 2.4]{hajcof}. Using again \cite[Lemma 2.4]{hajcof}, $\overline{R}$ is $\fa$-torsion free for all $\fa\in \Phi.$ We can take an element $x\in \fa$ which is regular on $\overline{R}.$ The short exact sequence
	$$0\to \overline{R}\to \overline{R}\to \overline{R}/x\overline{R}\to 0$$ deduces the following exact sequence
	$$\cdots\overset{f}{\to} \lc^{p+d-1}_\Phi(M,\overline{R}/x\overline{R})\overset{g}{\to} \lc^{p+d}_\Phi(M,\overline{R})\overset{.x}{\to} \lc^{p+d}_\Phi(M,\overline{R})\to 0.$$
	Since $\dim(\overline{R}/x\overline{R})\le d-1,$ by the inductive hypothesis $\lc^{p+d-1}_\Phi(M,\overline{R}/x\overline{R})$ is $\Phi$-cofinite Artinian. Consequently, $\Hom_R(R/\fa,\mathrm{Im}f)$ is finitely generated as it is a submodule of $$\Hom_R(R/\fa,H^{p+d-1}_\Phi(M,\overline{R}/x\overline{R})).$$ Moreover, $\mathrm{Im}f$ is an Artinian $R$-module.  It follows from \cite[Lemma 2.5 and Definition 2.1]{AB1} that $\Ext_R^i(R/\fa,\mathrm{Im}f)$ is finitely generated for all $i\ge 0.$ Therefore, $\Ext_R^i(R/\fa,\mathrm{Im}g)$ is finitely generated for all $i\ge 0.$ In particularly, $\Hom_R(R/\fa,\Image g)$ is finitely generated and
	\begin{align*}
	\Hom_R(R/\fa,\Image g)&=\Hom_R(R/\fa,0:_{\lc^{p+d}_\Phi(M,\overline{R})} x)\\
	&\cong \Hom_R(R/\fa,\lc^{p+d}_\Phi(M,\overline{R})).
	\end{align*}
	The proof above gives that $\lc^{p+d}_\Phi(M,\overline{R})$ is Artinian as $\dim{\overline{R}}\le d.$ By  \cite[Lemma 2.5 and Definition 2.1]{AB1},  $\lc^{p+d}_\Phi(M,\overline{R})$ is $\Phi$-cofinite.
	The short exact sequence
	$$0\to \Gamma_{\Phi}(R)\to R\to \overline{R}\to 0$$
	induces a long exact sequence
	$$\cdots\to \lc^{p+d}_\Phi(M,\Gamma_{\Phi}(R))\to \lc^{p+d}_\Phi(M,R)\to \lc^{p+d}_\Phi(M,\overline{R})\to 0.$$
	We get by \cite[Lemma 2.8]{hajcof} that $\lc^i_\Phi(M,\Gamma_{\Phi}(R))\cong \Ext_R^i(M,\Gamma_{\Phi}(R))$ for all $i\geq 0$ and then $\lc^i_\Phi(M,\Gamma_{\Phi}(R))$ is finitely generated for all $i\geq 0.$
	Therefore $\lc^{p+d}_\Phi(M,R)$ is $\Phi$-cofinite and the proof is complete.
\end{proof}

\begin{cor}\label{C:dimphicofart}
	Let $(R,\fm)$ be a local ring. Then $H^{\dim (R)}_{\Phi}(R)$ is $\Phi$-cofinite Artinian.
\end{cor}
When $\dim(M)\leq 2$ or $\dim (X)\leq 2, $ the cofiniteness of $H^i_I(M,X)$ was shown in \cite[Theorem 2.9]{mafcf} or \cite[Theorem 1.3]{CGH}. We also see in \cite[Corollary 2.15]{AB1} that if $(R,\fm)$ is a local ring with $\dim(R)\leq 2,$ then $\lc^i_\Phi(X)$ is $\Phi$-cofinite for all $i\geq 0.$ The following result might be considered as a generalization of these facts.
\begin{thm}\label{T:dimle2}
	Let $(R,\fm)$ be a local ring and $M,X$ two finitely generated $R$-modules with $\mathrm{pd}(M)<\infty.$ If $\dim(X)\le 2$ or $\dim (M)\le 2,$ then $\lc^{i}_{\Phi}(M,X)$ is $\Phi$-cofinite for  all $i\ge 0.$
\end{thm}
\begin{proof} There is a Grothendieck spectral sequence
	$$E_2^{p,q}=\lc^p_{\Phi}(\Ext^q_R(M,X))\underset{p}{\Rightarrow}\lc^{p+q}_{\Phi}(M,X).$$
	The assumption shows that $\dim(\Ext_R^i(M,X))\le 2$ and $\Ext_R^i(M,X)$ is finitely generated  for all $i\ge 0.$
	
	By Corollary \ref{C:dimphicofart},	$\lc^2_{\Phi}(\Ext^q_R(M,X))$ is $\Phi$-cofinite for all $q\ge 0.$  On the other hand	
	$\lc^p_{\Phi}(\Ext^q_R(M,X))=0$ for all $p>2,q\ge 0$ by
	\cite[ 2.7]{BZ2}.
	Let $q\ge 0,$ the homomorphisms of spectral sequence
	$$0=E_{2}^{i-2,q-i+1}\to E_{2}^{i,q-i}\to E_{2}^{i+2,q-i-1}=0$$
	induce
	$$E_{2}^{i,q-i}=E_\infty^{i,q-i}$$ 	where $i=0,1.$
	There is a filtration $\phi$ of submodules of $\lc^q=\lc^q_\Phi(M,X)$ such that
	$$E_\infty^{i,q-i} \cong \phi^i\lc^q/\phi^{i+1}\lc^q$$
	for all $0\le i\le q.$ 	Since $E_\infty^{i,q-i}$ is a sub-quotient of $E_2^{i,q-i},$ we claim that $E_\infty^{i,q-i}=0$ for all $i>2.$ Thus, 	$$\phi^3\lc^q=\phi^4\lc^q=\ldots=\phi^{q+1}\lc^q=0.$$
	The homomorphisms of spectral sequence
	$$0\to E_{2}^{0,q-1}\overset{d_2^{0,q-1}}{\to} E_{2}^{2,q-2}\overset{d_2^{2,q-2}}{\to} E_{2}^{4,q-3}=0$$
	turn out $$E_\infty^{2,q-2}=E_3^{2,q-2}=\Ker d_2^{2,q-2}/\Image d_2^{0,q-1}=E_{2}^{2,q-2}/\Image d_2^{0,q-1}.$$
	The short exact sequence
	$$0\to \Image d_2^{0,q-1}\to E_{2}^{2,q-2}\to E_\infty^{2,q-2}\to 0$$
	induces the following exact sequence
	$$\Hom_R(R/\fa,E_2^{2,q-2})\lo \Hom_R(R/\fa,E_\infty^{2,q-2})\lo \Ext_R^1(R/\fa,\Image d_2^{0,q-1}).$$
	Since $E_2^{2,q-2}$ is $\Phi$-cofinite Artinian by Corollary \ref{C:dimphicofart}, $\Hom_R(R/\fa,E_2^{2,q-2})$ is finitely generated. Furthermore, $\Image d_2^{0,q-1}$ is finitely generated.  Therefore, $\Hom_R(R/\fa,E_\infty^{2,q-2})$ is finitely generated. Note that $E_\infty^{2,q-2}$ is Artinian as it is a sub-quotient of the Artinian $R$-module $E_2^{2,q-2}=\lc^2_\Phi(\Ext_R^{q-2}(M,X))$ by Corollary \ref{C:dimphicofart}. It follows from \cite[Lemma 2.5 and Definition 2.1]{AB1} that $\Ext_R^i(R/\fa,E_\infty^{2,q-2})$ is finitely generated for all $i\ge 0.$
	
	The equality $E_\infty^{1,q-1}=E_{2}^{1,q-1}$ shows that	$\phi^1\lc^q/\phi^2\lc^q$ is $\Phi$-cofinite by \cite[Corolary 2.15]{AB1}. On the other hand, we have $\phi^2\lc^{q}\cong E_\infty^{2,q-2}$ and then
	$\Ext_R^i(R/\fa,\phi^1\lc^q)$	 is finitely generated for all $i\ge 0.$
	
	Note that 	$\phi^0\lc^q/\phi^1\lc^q\cong E_\infty^{0,q}$ is finitely generated since $E_\infty^{0,q}$ is a sub-quotient of $E_2^{0,q}$ which is finitely generated.	
	Thus $\Ext_R^i(R/\fa,\phi^0\lc^q)$ 	 is finitely generated for all $i\ge 0.$ This means that the module
	$\phi^0\lc^q=\lc^q_{\Phi}(M,X)$ is $\Phi$-cofinite.
\end{proof}
\begin{cor}\label{C:dimRle2}
	Let $(R,\fm)$ be a local ring with $\dim(R)\le 2.$ Then $\lc^{i}_{\Phi}(M,X)$ is $\Phi$-cofinite for  all $i\ge 0.$
\end{cor}

%%%%%%%%%%%%%%%%%%%%%%%%%%%%%%%%%%%%%%%%%%%%%%%%%%%%%%%%%%%%%%%

The following theorem is a generalization of \cite[Theorem 3.6]{BA}.

\begin{thm}\label{moh}
Let $R$ be a semi-local Noetherian ring and $\Phi$ a system of ideals of $R$. Let $M$ and $X$ be two $\fa$-$ETH$-cofinite $R$-module for all $\fa\in \Phi$ {\rm (}or finitely generated{\rm)}. If $\dim R/{\fa}\leq0$ for all $\fa\in \Phi$, then $\lc^{i}_{\Phi}(M,X)$ is $\Phi$-cofinite Artinian for all $i\ge 0.$
\end{thm}

\begin{proof}
Since  by \cite[Lemma 2.1]{BZ1}, $$\lc^{i}_\Phi(M)\cong\underset { \fa \in \Phi}
\varinjlim \lc^{i}_\fa(M),$$ it is easy to see that $\Supp_R(\lc^{i}_\Phi(M))
\subseteq\underset{\fa \in \Phi}\bigcup \Supp_R(\lc^{i}_\fa(M))$ and therefore
$$\dimSupp \lc^{i}_\Phi(M) \leq \sup \{ \dimSupp \lc^{i}_{\fa}(M)| \fa \in
\Phi \} \leq 0,$$ thus $\Supp_R(\lc^{i}_\Phi(M))\subseteq \Max(R)$. Now, it follows from \cite[Corolary 3.4]{BA} that $\lc^{i}_\Phi(M)$ is an Artinian $R$-module for all $i\geq 0$. Therefore Remark \ref{R:tinhchatminimax} (i) and Corollary \ref{nam} imply that $\lc^{i}_\Phi(M,X)$ is $\Phi$-cofinite for all $i\geq 0$.
\end{proof}

%%%%%%%%%%%%%%%%%%%%%%%%%%%%%%%%%%%%%%%%%%%%%%%%%%%%%%%%%%%%%%%

\begin{center}
{\bf Acknowledgments}
\end{center}

%The author likes to thank the referee for his/her careful reading and many helpful
%suggestions on this paper,
The authors are deeply
grateful to the referee for careful reading of the manuscript and for the helpful
suggestions.

%\newpage
%%% ----------------------------------------------------------------------
%%% ----------------------------------------------------------------------
%%% ----------------------------------------------------------------------
\bibliographystyle{amsplain}
%%% ----------------------------------------------------------------------
%%% ----------------------------------------------------------------------
%%% ----------------------------------------------------------------------

%%% ----------------------------------------------------------------------
%%% ----------------------------------------------------------------------
%%% ----------------------------------------------------------------------
\end{document}